\documentclass[11pt,sumlimits,intlimits]{amsart}
\usepackage{enumerate,mathtools}
\usepackage[backref]{hyperref}

\theoremstyle{plain}
\newtheorem{theorem}{Theorem}[section]

\newtheorem{lemma}[theorem]{Lemma}
\newtheorem{proposition}[theorem]{Proposition}
\theoremstyle{definition}
\newtheorem{definition}[theorem]{Definition}

\numberwithin{equation}{section}

\theoremstyle{remark}
\newtheorem{remark}[theorem]{Remark}
\input{xy} 
\xyoption{all} 
\CompileMatrices 
\DeclareMathOperator{\Ad}{Ad}
\newcommand{\eps}{\epsilon}
\newcommand{\id}{\mathrm{id}}

\begin{document}

\title{Decomposable approximations revisited}

\thanks{N.B. was partially supported by NSF grant DMS-1201385; J.C. by
  NSF Postdoctoral Fellowship DMS-1303884; S.W. by an Alexander von
  Humboldt foundation fellowship and by the DFG (SFB 878).}

\author[N.\ Brown]{Nathanial P.\ Brown}

\address{\hskip-\parindent Nathanial P.\ Brown. Department of
  Mathematics, The Pennsylvavia State University, University Park,
  State College, PA, 16802, USA.}
\email{nbrown@math.psu.edu}

\author[J.\ Carri\'{o}n]{Jos\'{e} R. Carri\'{o}n} 
\address{\hskip-\parindent Jos\'{e} R.\ Carri\'{o}n, Department of
  Mathematics, Texas Christian University, Fort Worth, Texas 76129,
  USA.}
\email{j.carrion@tcu.edu}

\author[S.\ White]{Stuart White}
\address{\hskip-\parindent Stuart White, School of Mathematics and
  Statistics, University of Glasgow, Glasgow, G12 8QW, Scotland and
  Mathematisches Institut der WWU M\"unster, Einsteinstra\ss{}e 62,
  48149 M\"unster, Germany.}
\email{stuart.white@glasgow.ac.uk}

\begin{abstract}
  Nuclear C$^*$-algebras enjoy a number of approximation properties,
  most famously the completely positive approximation property. This
  was recently sharpened to arrange for the incoming maps to be sums
  of order-zero maps. We show that, in addition, the outgoing maps can
  be chosen to be asymptotically order-zero.  Further these maps can
  be chosen to be asymptotically multiplicative if and only if the
  $\mathrm{C}^*$-algebra and all its traces are quasidiagonal.
\end{abstract} 

\date{\today}

\maketitle

\section{Introduction}

Approximation properties are ubiquitous in operator algebras,
characterizing many key notions and providing essential tools.  In
particular, and central to this note, a foundational result of
Choi-Effros \cite{CE:AJM} and Kirchberg \cite{K:MN} describes
nuclearity of a $\mathrm{C}^*$-algebra in terms of completely positive
approximations.  Precisely, $A$ is nuclear if and only if there exist
finite dimensional algebras $(F_i)$ and completely positive
contractive (c.p.c.{}) maps
\begin{equation}\label{BCW:e1.1}
  A \stackrel{\psi_i} \longrightarrow F_i
  \stackrel{\phi_i} \longrightarrow A
\end{equation}
that approximate the identity in the point-norm topology, i.e.
\begin{equation}\label{BCW:e1.2}
  \lim_i\|\phi_i(\psi_i(x))-x\|=0,\quad x\in A.
\end{equation}
Some 30 years later, via Connes' celebrated work on injective von
Neumann algebras \cite{C:Ann}, this approximation property was shown
to imply a stronger version of itself: one can always take every
$\phi_i$ to be a convex combination of contractive order-zero maps
(\cite{HKW:Adv}).  This has proved crucial to applications to near
inclusions (for example, \cite[Theorem~2.3]{HKW:Adv}). In this note we
observe a further improvement: every $\psi_i$ can be taken to be
asymptotically order zero, meaning that if $a,b \in A$ are
self-adjoint and $ab = 0$, then
\begin{equation}\label{BCW:e1.3}
  \lim_i\|\psi_i(a) \psi_i(b)\|=0.
\end{equation}
It was known that (\ref{BCW:e1.3}) could be arranged under the
stronger hypothesis of finite nuclear dimension
\cite[Proposition~3.2]{WZ:Adv} and this proved vital to various
applications (cf.\ \cite{BW:CRMASSRC,Robert11,W:Invent,W:Invent2}).

Our proof follows the strategy in \cite{HKW:Adv} by obtaining suitable
factorizations of the canonical inclusion $A\hookrightarrow A^{**}$
with respect to the weak$^*$ topology; then adjusting these to take
values in $A$; and finally applying a Hahn-Banach argument to get
asymptotic factorizations in the point-norm topology. To do this in
general, however, we require some quasidiagonal ideas. Indeed, the
main technical hurdle is showing that if $A$ is quasidiagonal and all
traces on $A$ are quasidiagonal in the sense of \cite{B:MAMS}, then
one can take every $\psi_i$ to be asymptotically multiplicative (see
Theorem~\ref{BCW:thm:QDfactorization}), while retaining the
decomposition of $\phi_i$ as a convex combination of contractive order
zero maps.  This should be compared with Blackadar and Kirchberg's
characterization of nuclear quasidiagonal $\mathrm{C}^*$-algebras in
\cite{BK:MA} as those with approximations (\ref{BCW:e1.1}) and
(\ref{BCW:e1.2}) with $\psi_i$ asymptotically multiplicative.

Since all traces on nuclear quasidiagonal $\mathrm{C}^*$-algebras in
the UCT class are quasidiagonal \cite{Tikuisis-White-etal15}, our
result improves the Blackadar-Kirchberg characterization in this
case. Cones over nuclear $\mathrm{C}^*$-algebras are quasidiagonal
\cite{V:DMJ} and satisfy the UCT, so all their traces are
quasidiagonal (though we show how Gabe's work \cite{G:JFA} gives a
simpler proof of this fact in Proposition
\ref{BCW:prop:traces-on-cones-qd}). Thus we obtain our main theorem
for general nuclear $A$ by taking an order-zero splitting $A \to CA$,
applying the improved approximation maps on $CA$, then using the
quotient map $CA \to A$ to get back to $A$ (see the proof of Theorem
\ref{BCW:thm:order0-approximations} for details).

\section{Quasidiagonal traces} 

In this note, a \emph{trace} on a $C^*$-algebra means a tracial
state. Write $T(A)$ for the collection of all traces on $A$. Various
approximation properties for traces were studied in \cite{B:MAMS}; of
particular relevance here is the notion of quasidiagonality for
traces.

\begin{definition}
  A trace $\tau$ on a $C^*$-algebra $A$ is \emph{quasidiagonal} if
  there exist finite dimensional algebras $F_i$, tracial states
  $\tau_i$ on $F_i$ and c.p.c.\ maps $\theta_i\colon A\to F_i$ such
  that $\mathrm{tr}_{i} \circ \theta_i \to \tau$ in the weak$^*$
  topology and
  \begin{equation}
    \lim_{i} \| \theta_i(ab) - \theta_i(a) \theta_i(b) \|
    = 0
  \end{equation}
  for all $a, b \in A$.  Write $T_{\mathrm{qd}}(A)$ for the set of
  quasidiagonal traces of $A$.
\end{definition}
When $A$ is unital the maps $\theta_i$ can be taken to be unital and
completely positive (u.c.p.).  Theorem~3.1.6 of \cite{B:MAMS} lists
several other characterizations of amenable traces.

The main technical result of this note is the following. 

\begin{theorem} 
  \label{BCW:thm:QDfactorization}
  Let $A$ be a separable and nuclear $C^*$-algebra. Then $A$ is
  quasidiagonal and $T(A) = T_{\mathrm{qd}}(A)$ if and only if there
  exist a sequence of finite-dimensional $C^*$-algebras $(F_n)$ and
  c.p.c. maps
\begin{equation}
    A \xrightarrow{\psi_n} F_n\xrightarrow{\phi_n} A
\end{equation}
  such that
  \begin{enumerate}
  \item $\| (\phi_n\circ \psi_n) (a) - a \| \to 0$ for all $a\in
    A$;\label{BCW:thm:QDfactorization.Condition1}
    
  \item every $\phi_n$ is a convex combination of finitely many
    contractive order zero maps;
    and\label{BCW:thm:QDfactorization.Condition2}
    
  \item $\| \psi_n(ab) - \psi_n(a) \psi_n(b) \| \to 0$ for all $a,
    b\in A$.\label{BCW:thm:QDfactorization.Condition3}
\end{enumerate}
\end{theorem} 

Only one implication of this theorem requires much work. Indeed, if
$A$ has approximations with properties
(\ref{BCW:thm:QDfactorization.Condition1})--(\ref{BCW:thm:QDfactorization.Condition3}),
then $A$ is quasidiagonal (this is an easy implication in
\cite[Theorem 5.2.2]{BK:MA}; the maps $\psi_i$ are approximately
multiplicative by (\ref{BCW:thm:QDfactorization.Condition3}), and
(\ref{BCW:thm:QDfactorization.Condition1}) ensures that they are
approximately isometric). It is equally routine to check that all
traces are quasidiagonal.  Indeed, since a trace composed with an
order-zero map is a trace by \cite[Corollary
4.4]{Winter-Zacharias09b}, and each $\phi_n$ is a convex combination
of order zero maps, given a trace $\tau_A\in T(A)$, it follows that
$\tau_A\circ\phi_n$ defines a trace on $F_n$. Then condition
(\ref{BCW:thm:QDfactorization.Condition1}) ensures that
$\tau_A\circ\phi_n\rightarrow \tau_A$ weak$^*$.

In order to prove the reverse implication it will suffice to prove a
$\sigma$-weak version for the canonical inclusion $\iota \colon A
\hookrightarrow A^{**}$.  Namely, we prove the following proposition
in the remainder of this section.

\begin{proposition} 
  \label{BCW:prop:weak-factorization}
  Let $A$ be a separable nuclear, and quasidiagonal $C^*$-algebra
  with $T(A) = T_{\mathrm{qd}}(A)$. Then there are nets of
  finite-dimensional $C^*$-algebras $(F_i)$ and of c.p.c. maps
  \begin{equation}
    A \xrightarrow{\psi_i} F_i \xrightarrow{\phi_i} A^{**}
  \end{equation}
  such that
  \begin{enumerate} \label{BCW:prop:weak-factorization.Condition1}
  \item $(\phi_i \circ \psi_i)(a) \to \iota(a)$ in the $\sigma$-weak
    topology for every $a\in A$;
  \item $\phi_i$ is an order zero
    map; \label{BCW:prop:weak-factorization.Condition2}
  \item $\| \psi_i(ab) - \psi_i(a) \psi_i(b) \| \to 0$ for every $a,
    b\in A$. \label{BCW:prop:weak-factorization.Condition3}
  \end{enumerate}
\end{proposition}

With this proposition in hand we prove
Theorem~\ref{BCW:thm:QDfactorization} by following the same steps used
to prove \cite[Theorem 1.4]{HKW:Adv} from the preparatory lemma
\cite[Lemma 1.3]{HKW:Adv}.  Indeed, using the notation of
Proposition~\ref{BCW:prop:weak-factorization}, first apply Lemma~1.1
of \cite{HKW:Adv} to see that for every $i$ there is a net of
contractive order zero maps $(\phi_{i,\lambda}\colon F_i\to
A)_\lambda$ such that $\phi_{i, \lambda}(x)$ converges $\sigma$-weakly
to $\phi_i(x)$ for every $x\in F_i$.  We may therefore assume that the
image of $\phi_i$ is contained in $A$ for every $i$.  The argument now
ends with a familiar Hahn-Banach argument, similar to the one used to
prove the completely positive approximation property of a
$C^*$-algebra from the assumption that its enveloping von Neumann
algebra is semidiscrete (see \cite[Proposition 2.3.8]{Brown-Ozawa08}).
Briefly, given a finite subset $\mathcal F$ of $A$ and $\epsilon>0$,
let $K_0 \subset \mathcal{B}(A)$ be the collection of all c.p.c maps
$\theta\colon A\to A$ which factorize as
$A\stackrel{\psi}{\rightarrow}F\stackrel{\phi}{\rightarrow}A$, where
$\psi$ is a c.p.c.\ map with $\|\psi(ab)-\psi(a)\psi(b)\|\leq
\epsilon$ for all $a,b\in \mathcal F$, and $\phi$ is a contractive
order zero map.  Since the identity map on $A$ lies in the point-weak
closure of $K_0$, it lies in the point norm closure of the convex hull
of $K_0$.  As a convex combination of maps in $K_0$ can be factorized
in the form
$A\stackrel{\psi}{\rightarrow}F\stackrel{\phi}{\rightarrow}A$, where
$\psi$ is a c.p.c.\ map with $\|\psi(ab)-\psi(a)\psi(b)\|\leq
\epsilon$ for all $a,b\in \mathcal F$ and $\phi$ a convex combination
of contractive order zero maps, we can find such $\psi$ and $\phi$
additionally satisfying $\|\phi(\psi(a))-a\|<\epsilon$ for
$a\in\mathcal F$. Theorem~\ref{BCW:thm:QDfactorization} follows by
using a countable dense subset of $A$ to produce the required sequence
of maps.

The proof of Proposition~\ref{BCW:prop:weak-factorization} requires some
lemmas and will first be carried out in the case when $A$ is
unital. We will split $A^{**}$ into two pieces, the finite and
properly infinite summands, and then handle each piece
separately.\footnote{%
  Recall that a von Neumann algebra is finite if it admits a
  separating family of tracial states, and properly infinite if it has
  no finite summand.} %
The properly infinite case is handled by a combination of Blackadar
and Kirchberg's characterization of NF-algebras in \cite{BK:MA} and
Haagerup's very short proof that semidiscreteness implies
hyperfiniteness for properly infinite von Neumann algebras
\cite[Section 2]{H:JFA}.

Recall that if $\rho$ is a normal state on a von Neumann algebra $M$,
the seminorm $\|\cdot\|_\rho^\sharp$ is given by
\begin{equation}
  \|x\|_\rho^\sharp
  = \rho\left(\frac{xx^*+x^*x}{2}\right)^{1/2},
  \quad x\in M.
\end{equation}
It is a standard fact (see e.g. \cite[III.2.2.19]{Blackadar06}) that
if $\{ \rho_i \}$ is a separating family of normal states on $M$, then
the topology generated by $\{ \|\cdot\|^\sharp_{\rho_i} \}$ agrees
with the $\sigma$-strong$^*$ topology on any bounded subset of $M$.
This will be used in both of the following lemmas.

\begin{lemma}
  \label{BCW:lem:properlyinfinite}
  Let $A$ be a unital, quasidiagonal and nuclear
  $\mathrm{C}^*$-algebra. Let $\pi_\infty:A\rightarrow M$ be the
  properly infinite summand of the universal representation of
  $A$. Then there are nets of finite-dimensional $C^*$-algebras $F_i$
  and nets of c.p.c. maps
  \begin{equation}
    A \xrightarrow{\psi_i} F_i \xrightarrow{\phi_i} M
  \end{equation}
  such that
  \begin{enumerate}
  \item $(\phi_i \circ \psi_i)(a) \to \pi_\infty(a)$ in the
    $\sigma$-strong$^*$ topology (and hence also in the $\sigma$-weak
    topology) for every $a\in
    A$; \label{BCW:lem:properlyinfinite.Condition1}

  \item $\phi_i$ is a $^*$-homomorphism for every $i$;
    and\label{BCW:lem:properlyinfinite.Condition2}

  \item $\| \psi_i(ab) - \psi_i(a) \psi_i(b) \| \to 0$ for every $a,
    b\in A$.\label{BCW:lem:properlyinfinite.Condition3}
  \end{enumerate}
\end{lemma}

\begin{proof}
  Fix $\epsilon>0$, a finite subset $\mathcal F$ of unitaries in $A$, and
  finitely many normal states $\rho_1,\dots,\rho_m$ on $M$.  We will
  produce a factorization
  \begin{equation}
    A \xrightarrow{\psi} F \xrightarrow{\phi} M
  \end{equation}
  where $F$ is a matrix algebra, $\phi$ is a $^*$-homomorphism and
  $\psi$ is a u.c.p.\ map, such that
  \begin{align}
    \|\phi(\psi(u))-u\|_{\rho_i}^\sharp&<2\epsilon^{\frac12}\label{BCW:e2.7}\\
    \shortintertext{and}
    \|\psi(uv)-\psi(u)\psi(v)\|&<\epsilon,
  \end{align}
  for all $u,v\in\mathcal F$ and $i=1,\dots,m$.  In this way we obtain
  the desired net indexed by finite subsets of unitaries, finite
  subsets of normal states and tolerances $\epsilon$. By working with
  $\rho=\frac{1}{m}\sum_{i=1}^m\rho_i$, and replacing $\epsilon$ by
  $\epsilon/m$, it suffices to obtain the single estimate
  \begin{equation}
    \|\phi(\psi(u))-u\|_{\rho}^\sharp<2\epsilon^{\frac12},
    \quad u\in\mathcal F,
  \end{equation}
  in place of (\ref{BCW:e2.7}).

  Since $A$ is nuclear and quasidiagonal, it is NF by \cite[Theorem
  5.2.2]{BK:MA} and so, by part (vi) of this theorem, there exists a
  matrix algebra $F$ and u.c.p.\ maps
  \begin{equation}
    A \xrightarrow{\psi} F \xrightarrow{\theta} A
  \end{equation}
  such that
  \begin{equation}
    \|(\theta \circ \psi)(u)-u\| < \epsilon \label{BCW:Infinite.1}
  \end{equation}
  and
  \begin{equation}
    \| \psi(uv) - \psi(u) \psi(v) \| < \epsilon,
  \end{equation}
  for all $u,v\in \mathcal F$.  The estimate in (\ref{BCW:Infinite.1})
  gives
  \begin{equation}\label{BCW:eq16}
    \|\pi_\infty(\theta(\psi(u))-u)\|_\rho^\sharp<\epsilon,
  \end{equation}
  for all $u\in\mathcal F$.
  
  We now follow the proof of \cite[Theorem 2.2]{H:JFA}. As $M$ is
  properly infinite, we can fix a unital embedding $\iota:F\rightarrow
  M$. Then by \cite[Proposition 2.1]{H:JFA} there exists an isometry
  $v\in M$ such that $\theta(x)=v^*\iota(x)v$ for all $x\in F$.  If
  $v$ is a unitary (which is impossible, in general), then we're done
  because $\mathrm{Ad}(v) \circ \iota$ is the desired
  $*$-homomorphism. Since the $\sigma$-strong closure of unitaries in
  any von Neumann algebra is the set of all isometries (cf.\
  \cite[Lemma XVI.1.1]{Takesaki03a}), the remainder of the proof
  (which follows the estimates on page 167 of \cite{H:JFA}) amounts to
  approximating $v$ with a suitable unitary.

  We may assume that $M$ is concretely represented on some Hilbert
  space $\mathcal H$ so that $\rho$ is a vector state, given by a unit
  vector $\xi\in\mathcal H$. Using the identity $\|x\xi\|^2 +
  \|x^*\xi\|^2 = 2(\|x\|_\rho^\sharp)^2$, which is valid for all $x\in
  M$, and equation (\ref{BCW:eq16}) we have
 \begin{align}
    \| ( v^*\iota( \psi(u) )v - \pi_\infty(u) )\xi \|
    <
    2^{\frac12}\epsilon\\
    \shortintertext{and}
    \| ( v^*\iota(\psi(u)^*)v-\pi_\infty(u^*) )\xi \|
    <
    2^{\frac12}\epsilon.
  \end{align}
 This implies 
  \begin{align}
    \Re \big\langle \iota( \psi(u) )v\xi , v\pi_\infty(u)\xi \big\rangle
    >
    1 - 2^{\frac12}\epsilon\\
    \shortintertext{and}
    \Re\big\langle \iota( \psi(u)^* )v\xi , v\pi_\infty(u^*)\xi \big\rangle
    >
    1-2^{\frac12}\epsilon.
  \end{align}
 Now choose a unitary $w\in M$ such that, for all $u\in
  \mathcal{F}$,
  \begin{align}
    \Re\big\langle \iota( \psi(u) )w\xi , w\pi_\infty(u)\xi \big\rangle
    >
    1-2\epsilon\\
    \shortintertext{and}
    \Re\big\langle \iota( \psi(u)^* )w\xi , w\pi_\infty(u^*)\xi \big\rangle
    >
    1-2\epsilon.
  \end{align}
  Then, since $\|\iota(\psi(u))w\xi\|\leq 1$ and
  $\|\iota(\psi(u^*))w\xi\|\leq 1$, we have
  \begin{align}
    \| \iota( \psi(u) )w\xi - w\pi_\infty(u)\xi \|^2
    &\leq
      2 - 2\Re\langle \iota( \psi(u) )w\xi , w\pi_\infty(u)\xi \rangle
      < 4\epsilon\\    \shortintertext{and}
    \| \iota( \psi(u^*) )w\xi - w\pi_\infty( u^* )\xi \|^2
    &\leq
      2 - 2\Re\langle \iota( \psi(u^*) )w\xi , w\pi_\infty( u^* )\xi\rangle
      < 4\epsilon,
  \end{align}
  for all $u\in\mathcal F$.  Then
  $\phi=\mathrm{Ad}(w^*)\circ\iota:F\rightarrow M$ is a
  $^*$-homomorphism with
  \begin{equation}
    \| \phi(\psi(u)) - \pi_\infty(u) \|_\rho^\sharp
    <
    (4\epsilon)^{\frac12},
    \quad u\in\mathcal F,
  \end{equation}
  as required.
\end{proof}

Next we deal with the finite part of $A^{**}$. We need the following
standard uniqueness fact.  Let $A$ be a separable nuclear
$C^*$-algebra, and $N$ a finite von Neumann algebra. Then it is well
known, though most often stated when $N$ is a factor (see
\cite{Jung07} and \cite{Atkinson15} which give converse statements),
or when $N$ has separable predual (see \cite[Theorem
5]{Ding-Hadwin05}) that two $^*$-homomorphisms
$\phi_1,\phi_2:A\rightarrow N$ are $\sigma$-strong$^*$ approximately
unitarily equivalent in that there is a net of unitaries $u_i$ such
that $u_i\phi_1(a)u_i^*\rightarrow \phi_2(a)$ in the
$\sigma$-strong$^*$ topology for all $a\in A$ if and only if
$\tau\circ\phi_1=\tau\circ\phi_2$ for all normal traces $\tau$ on
$N$. Indeed, $\phi_1$ and $\phi_2$ extend to normal representations
$\phi_1^{**},\phi_2^{**}:A^{**}\rightarrow N$ that agree on
traces. Since $A^{**}$ is injective, it is hyperfinite\footnote{%
  See \cite{Elliott78} for the extension of Connes' theorem to the
  non-separable predual case used here.}, %
so there is an increasing net of finite dimensional subalgebras
$(F_\lambda)$ that is $\sigma$-strong$^*$ dense in $A^{**}$.  For each
$\lambda$, the condition that
$\tau\circ\phi_1^{**}|_{F_\lambda}=\tau\circ\phi_2|^{**}_{F_\lambda}$
for all normal traces $\tau$ on $N$ gives a unitary $u_\lambda$ with
$\mathrm{Ad}(u_\lambda)\circ\phi_1^{**}|_{F_\lambda}=\phi_2^{**}|_{F_\lambda}$.
The net of unitaries $(u_\lambda)$ witnesses the $\sigma$-strong$^*$
approximate unitary equivalence of $\phi_1^{**}$ and $\phi_2^{**}$ and
hence also of $\phi_1$ and $\phi_2$.

\begin{lemma}
  \label{BCW:lem:weak-factorization-finite}
  Let $A$ be a separable, unital and nuclear $C^*$-algebra and assume
  $T(A) = T_{\mathrm{qd}}(A)$.  Let $\pi_{\mathrm{fin}}\colon A\to M$
  be the finite summand of the universal representation of $A$. Then
  there are nets of finite dimensional $C^*$-algebras $F_i$ and of
  c.p.c. maps
\begin{equation}
    A \xrightarrow{\psi_i} F_i \xrightarrow{\phi_i} M
\end{equation}
  such that
  \begin{enumerate}
  \item $(\phi_i \circ \psi_i)(a) \to \pi_{\mathrm{fin}}(a)$ in the
    $\sigma$-strong$^*$ topology (and therefore also in the
    $\sigma$-weak topology) for every $a\in
    A$; \label{BCW:lem:weak-factorization-finite.Condition1}
    
  \item $\phi_i$ is a $^*$-homomorphism for every $n$;
    and \label{BCW:lem:weak-factorization-finite.Condition2}
    
  \item $\| \psi_i(ab) - \psi_i(a) \psi_i(b) \| \to 0$ for every $a,
    b\in A$.  \label{BCW:lem:weak-factorization-finite.Condition3}
  \end{enumerate}
\end{lemma}

\begin{proof}
  Recall that $M$ has a separating family of normal tracial states.
  As pointed out in the remarks preceding
  Lemma~\ref{BCW:lem:properlyinfinite}, on any bounded subset of $M$
  the $\sigma$-strong$^*$ topology agrees with the topology generated
  by the family of seminorms $\{ \| \cdot \|_{2,\tau} \}$ (where
  $\tau$ runs through all normal tracial states of $M$).  As in the
  proof of Lemma~\ref{BCW:lem:properlyinfinite}, the required nets of
  finite dimensional $C^*$-algebras and c.p.c. maps will ultimately be
  indexed by finite subsets $\mathcal{F}$ of $A$, positive numbers
  $\epsilon$, and finite subsets $\{\tau_1, \dots, \tau_m\}$ of normal
  tracial states of $M$.  Moreover, the same argument found in the
  proof of Lemma~\ref{BCW:lem:properlyinfinite} shows that it suffices
  to consider a single normal trace $\tau$ (by considering $\tau =
  \frac{1}{m}\sum_{i=1}^m \tau_i$), which we fix for the remainder of
  the proof.

  Write $N$ for $\pi_\tau(A)''$.  We claim it is enough to obtain
  finite dimensional algebras $F_i$ and maps $\psi_i\colon A\to F_i$
  and $\phi_i\colon F_i\to N$ (as opposed to $\phi_i\colon F_i\to M$)
  satisfying (\ref{BCW:lem:weak-factorization-finite.Condition2}),
  (\ref{BCW:lem:weak-factorization-finite.Condition3}), and
  \begin{equation}
    \| (\phi_i \circ \psi_i)(a) - \pi_\tau(a) \|_{2, \tau} \to 0.
  \end{equation}
  For this, first note that $J = \{ x\in M : \tau(x^*x) = 0 \}$ is a
  (closed, two-sided) ideal of $M$, and therefore of the form $Mp$ for
  some central projection $p\in M$.  Using the fact that $\tau$ is a
  faithful trace on both $N$ and $M(1-p)$, we get that $N\cong M(1-p)$
  (extending the identity on $A/ J\cap A$).  Identifying $N$ with this
  direct summand, it follows that $\| \pi_\tau(a) -
  \pi_{\mathrm{fin}}(a) \|_{2, \tau} = 0$, which proves the claim.

  Being finite, $N$ is the direct sum of a (finite) type I von Neumann
  algebra and type II$_1$ von Neuman algebra.  We can therefore deal
  with each summand separately, and combine the two approximations to
  prove the lemma.  To ease the notation, we may as well assume that
  $N$ itself is type I or type II$_1$.

  First assume $N$ is finite type I, so of the form $N \cong \oplus_i
  L^{\infty}(X_i) \otimes M_{n_i}$ for some $n_i\in\mathbb N$ and
  measure spaces $X_i$. Write $\pi_\tau(a)=\oplus_i
  \pi_\tau^{(i)}(a)$. If the direct sum is infinite then, by normality
  of $\tau$, $\pi_\tau(a)$ is the limit in $\|\cdot\|_{2,\tau}$ of the
  finite sums $\oplus_{i=1}^n\pi_\tau^{(i)}(a)$, and so it suffices to
  prove the result when the sum $N \cong \oplus_i L^{\infty}(X_i)
  \otimes M_{n_i}$ is finite. In this case $N$ is a (non-separable) AF
  $\mathrm{C}^*$-algebra, so given a finite subset $\mathcal F$ of the
  unit ball of $N$ and $\epsilon>0$ there exists some finite
  dimensional $\mathrm{C}^*$-subalgebra $F\subset N$ such that for
  each $x\in \mathcal F$, there exists a contraction $y_x\in F$ with
  $\|x-y_x\|<\epsilon$.  Fix any conditional expectation $\psi:N\to F$
  (an expectation exists by Arveson's Extension Theorem) and note that
  for $x_1,x_2\in\mathcal F$
  \begin{align}
    \|\psi(x_1x_2)-\psi(x_1)\psi(x_2)\| & \leq \|x_1x_2 -
                                          y_{x_1}y_{x_2}\| +
                                          \|x_1-y_{x_1}\| + \|x_2-y_2\|
                                          \nonumber\\
                                        & \leq 4\epsilon.
  \end{align}
  Also, $\psi$ composed with the inclusion map $\phi\colon F
  \hookrightarrow N$ is the identity on $F$, so that
  $\|\phi(\psi(x))-x\|\leq 2\epsilon$ for $x\in\mathcal F$.  Thus the
  required approximations exist in the finite type I case.

  Assume now that $N$ is type II$_1$.  The center $Z(N)$ of $N$ is an
  abelian von Neumann algebra with faithful normal state $\tau$, so of
  the form $L^\infty(X,\mu)$, where $\mu$ is induced by $\tau$.  Let
  $E\colon N\rightarrow L^\infty(X,\mu)$ denote the center valued
  trace.  Let $(a_j)_{j=1}^\infty$ be a sequence of positive
  contractions in $A$ that is dense in the unit ball of $A_+$ and such
  that $\|a_j\|<1$ for all $j$.
  
  Fix $k\in\mathbb N$. Given a $k$-tuple
  $i=(i_1,\dots,i_k)\in\{1,\dots,k\}^k$, let $p_i$ be the projection
  in $L^\infty(X,\mu)$, whose characteristic function is the set
  \begin{equation}
    \{x\in X :
    \frac{i_j-1}{k}
    \leq
    E(\pi_\tau(a_j))(x)
    <
    \frac{i_j}{k},\ j=1,\dots,k\}.
  \end{equation}
  These are pairwise orthogonal and $\sum_ip_i=1_N$.  Some of the
  $p_i$ may be zero; in what follows we only work with and sum over
  those indices $i$ for which $p_i\neq 0$. Note that
  \begin{equation}
    \| E(\pi_\tau(a_j)) - \sum_{i}\frac{i_j}{k}p_i
    \|_{L^\infty(X,\mu)}
    \leq \frac{1}{k},\ j=1,\dots,k.
  \end{equation}
  Now, any normal trace on $N$ is of the form
  $\tau(f\cdot)$ for some $f\in L^1(X,\mu)_+$ with
  $\|f\|_{L^1(X,\mu)}=1$. For such an $f$,
  \begin{equation}\label{BCW:eq:4}
    \tau(f\pi_\tau(a_j))
    =
    \tau\big( fE(\pi_\tau(a_j)) \big)
    \approx_{\tfrac{1}{k}}
    \sum_{i}\tfrac{i_j}k\tau(fp_i),
    \quad j=1,\dots,k.
  \end{equation}
  Also, for each index $i$,
  \begin{equation}\label{BCW:SAW.N1}
    |\tau(p_i\pi_\tau(a_j)) - \tau(p_i)\tfrac {i_j}k|
    \leq
    \frac{1}{k}\tau(p_i),
    \quad j=1,\dots,k.
  \end{equation}

  Now, for each $i=(i_1,\dots,i_k)$, the map
  $\frac{1}{\tau(p_i)}\tau(\pi_\tau(\cdot) p_i)$ is a tracial state on
  $A$.  Because all traces on $A$ are quasidiagonal, there exist
  matrix algebras $F_{k,i}$ and u.c.p.\ maps $\psi_{k,i}\colon
  A\rightarrow F_{k,i}$ such that
  \begin{align}
    \Big|
    \mathrm{tr}_{F_{k,i}}\big( \psi_{k,i}(a_j) \big)
    -
    \frac{1}{\tau(p_i)} \tau\big( p_i\pi_\tau(a_j) \big)
    \Big|
    & < \frac{1}{k},
      \quad j=1,\dots,k \label{BCW:SAW.N2}\\
    \shortintertext{and}
    \| \psi_{k,i}(a_{j_1}a_{j_2})
    - \psi_{k,i}(a_{j_1}) \psi_{k,i}(a_{j_2}) \|
    & <\epsilon,
      \quad j_1,j_2=1,\dots,k.
  \end{align}
  Combining (\ref{BCW:SAW.N2}) and (\ref{BCW:SAW.N1}) gives
  \begin{equation}\label{BCW:SAW.N3}
    \Big| \mathrm{tr}_{F_{k,i}}\big( \psi_{k,i}(a_j) \big)-\frac{i_j}{k} \Big|
    \leq
    \frac{2}{k}.
  \end{equation}
  Define $F_k:=\bigoplus_i F_{k,i}$ and $\psi_k := \oplus \psi_{k,i}$
  so that (\ref{BCW:lem:weak-factorization-finite.Condition3}) holds.  Since each $p_iNp_i$ is type II$_1$, there
  exists a unital $^*$-homomorphism $\phi_{k,i}:F_{k,i}\rightarrow
  p_iNp_i$ (see e.g. \cite[Lemma 2.4.8]{Brown-Ozawa08}).  Define
  $\phi_k:F_k\rightarrow N$ by $\phi_k=\oplus_i\phi_{k,i}$.  This is a
  unital $^*$-homomorphism.  Further, for each $f\in L^1(X,\mu)_+$
  with $\|f\|_{L^1(X,\mu)}=1$, we have
  \begin{align}
    \tau\big( f\phi_k(\psi_k(a_j)) \big)
    & \stackrel{\hphantom{(\ref{BCW:SAW.N3})}}{_{\hphantom{\frac 2k}}=_{\hphantom{\frac2k}}} 
      \sum_i\tau(fp_i)\mathrm{tr}_{F_{k,i}}(\psi(a_j))\nonumber\\
    &\stackrel{(\ref{BCW:SAW.N3})}{_{\hphantom{\frac 2k}}\approx_{\frac 2k}}
      \sum_i\tau(fp_i)\tfrac{i_j}{k}\nonumber\\
    &\stackrel{\eqref{BCW:eq:4}}{_{\hphantom{\frac 2k}}\approx_{\frac{1}{k}}}
      \tau(f\pi_\tau(a_j)),\quad j=1,\dots,k.
  \end{align}
  
  Thus the sequence of maps $(\phi_k\circ\psi_k)$ satisfies
  \begin{equation}
    \lim_{k\to \infty} \
    \sup_{\stackrel{f\in L^1(X,\mu)_+}{\|f\|_{L^1(X,\mu)} = 1}}
    \big|\tau\big( f\phi_k(\psi_k(a_j)) \big)
    - \tau\big( f\pi_\tau(a_j) \big) \big| = 0,\quad j\in\mathbb N.
  \end{equation}
  
  Write $N^\omega$ for the ultraproduct of $N$ with respect to some
  fixed free ultrafilter $\omega\in\beta\mathbb N\setminus\mathbb N$
  (defined with respect to $\tau$). We claim that the sequence
  $(\phi_k\circ\psi_k)$ induces a $^*$-homomorphism, call it
  $\theta\colon A\to N^\omega$, that agrees on traces with
  $\pi_\tau^\omega$ (the composition of $\pi_\tau$ with the canonical
  embedding of $N$ into $N^\omega$).  Indeed, this follows as in the
  proof of Lemma 3.21 of \cite{BBSTWW:MAMS}: fix $j$ and write
  $x_{j,k}=\phi_k(\psi_k(a_j))-\pi_\tau(a_j)$. As
  $E(x_{j,k}-E(x_{j,k}))=0$, \cite[Theorem 3.2]{Fack-Harpe80} gives
  $y_{j,k,l}$ and $z_{j,k,l}$ in $N$ for $l=1,\dots,10$ such that
  \begin{equation}
    x_{j,k} - E(x_{j,k})
    =
    \sum_{l=1}^{10}[y_{j,k,l},z_{j,k,l}]
 \end{equation}
  with $\|y_{j,k,l}\|\leq 12 \|x_{j,k}-E(x_{j,k})\|$ and
  $\|z_{j,k,l}\|\leq 12$.  These estimates ensure that $(y_{j,k,l})_k$
  and $(z_{j,k,l})_k$ represent elements $y_{j,l}$ and $z_{j,l}$ in
  $N^\omega$.  Since
  \begin{equation}
    \|E(x_{j,k})\|
    =
    \sup_{\stackrel{f\in L^1(X,\mu)_+}{\|f\|_{L^1(X,\mu)} = 1}}
    \big|\tau\big( f\phi_k(\psi_k(a_j)) \big)
    - \tau\big( f\pi_\tau(a_j) \big)\big|,
  \end{equation}
  it follows that $(E(x_{j,k}))_k$ represents $0\in N^\omega$ and so
  $(x_{j,k})_k$ represents the finite sum of commutators
  $\sum_{l=1}^{10}[y_{j,l},z_{j,l}]$ in $N^\omega$ and hence is zero
  in all traces on $N^\omega$.

  By the remark preceding the lemma, $\theta$ and $\pi_\tau^\omega$
  are $\sigma$-strong$^*$ approximately unitarily equivalent.  Because
  $A$ is separable and we work in an ultrapower, a standard reindexing
  argument (using Kirchberg's $\epsilon$-test from \cite[Appendix
  A]{Kirchberg06}) shows that $\theta$ and $\pi_\tau^\omega$ are
  actually unitarily equivalent.  That is, there exists a sequence
  $(u_k)$ of unitaries in $N$ such that
  \begin{equation}
    \lim_{k\rightarrow\omega}
    \| u_k (\phi_k\circ\psi_k)(a) u_k^* -  \pi_\tau(a) \|_{2, \tau}
    = 0,\quad a\in A.
  \end{equation}

  Let $\tilde{\phi}_k = \Ad u_k \circ \phi_k$.  Passing to a
  subsequence, if necessary, we obtain
  \begin{equation}
    \lim_{k\rightarrow\infty}
    \| (\tilde{\phi}_k\circ\psi_k)(a) -  \pi_\tau(a) \|_{2, \tau}
    = 0,\quad a\in A,
  \end{equation}
  as was to be proved.
\end{proof}

\begin{proof}[{Proof of Proposition \ref{BCW:prop:weak-factorization}.}]
  For unital $C^*$-algebras, one just takes direct sums of the maps
  provided by Lemmas \ref{BCW:lem:properlyinfinite} and
  \ref{BCW:lem:weak-factorization-finite}.  The non-unital case
  follows from the unital case as follows.

  Assume $A$ is non-unital and $T(A) = T_{\mathrm{qd}}(A)$. Then by
  \cite[Proposition 3.5.10]{B:MAMS} we have $T(\tilde{A}) =
  T_{\mathrm{qd}}(\tilde{A})$, too, where $\tilde{A}$ is the
  unitization of $A$. Hence we can find nets of finite-dimensional
  $C^*$-algebras $(F_i)$ and c.p.c. maps
  \begin{equation}
    \tilde{A} \xrightarrow{\psi_i} F_i \xrightarrow{\phi_i} (\tilde{A})^{**}
  \end{equation}
  such that
  \begin{enumerate}
  \item $(\phi_i \circ \psi_i) (a) \to \iota_{\tilde{A}}(a)$, in the
    $\sigma$-weak topology for all $a\in \tilde{A}$
  \item every $\phi_i$ is a convex combination of finitely many
    contractive order zero maps; and
  \item $\| \psi_i(ab) - \psi_i(a) \psi_i(b) \| \to 0$ for all $a,
    b\in \tilde{A}$.
  \end{enumerate}
  The short exact sequence $0 \to A \to \tilde{A} \to \mathbb{C} \to
  0$ induces a canonical isomorphism $(\tilde{A})^{**} \cong A^{**}
  \oplus \mathbb{C}$. The desired maps are now gotten by restricting
  each $\psi_i$ to $A$ and using the $\sigma$-weakly continuous
  projection $(\tilde{A})^{**} \to A^{**}$ to push the $\phi_i$'s back
  into $A^{**}$.
\end{proof}

\section{The main theorem} 

\begin{theorem} 
\label{BCW:thm:order0-approximations}
Let $A$ be a nuclear $C^*$-algebra. Then there exist nets of
finite-dimensional $C^*$-algebras $(F_i)$ and c.p.c. maps
\begin{equation}
  A \xrightarrow{\psi_i} F_i \xrightarrow{\phi_i} A
\end{equation}
such that
\begin{enumerate}
\item $\| (\phi_i \circ \psi_i) (a) - a \| \to 0$ for all $a\in A$;
\item every $\phi_i$ is a convex combination of finitely many
  contractive order zero maps;
\item $\| \psi_i(a) \psi_i(b) \| \to 0$ for all $a, b\in A_+$ that
  satisfy $ab = 0$.
\end{enumerate}
\end{theorem} 

To prove Theorem~\ref{BCW:thm:order0-approximations} we will apply
Theorem~\ref{BCW:thm:QDfactorization} to the cone $CA = C_0(0,1]
\otimes A$ of $A$.  We will need to know that all traces on $CA$ are
quasidiagonal for nuclear $A$.  While this follows from
\cite[Corollary 6.1]{Tikuisis-White-etal15},\footnote{The cone $CA$ is
  quasidiagonal by \cite{V:DMJ} and satisfies the UCT, since it is
  contractible.} it is really the case that the required statement is
a recasting of the ``order zero quasidiagonality result'' of
\cite[Proposition 3.2]{Sato-White-etal14} used as the starting point
in \cite{Tikuisis-White-etal15}. More generally, Gabe's ``order zero
quasidiagonality'' of amenable traces (\cite[Proposition 3.6]{G:JFA})
can also be expressed in this language, as set out below.

\begin{proposition}[{Gabe, c.f. \cite[Proposition 3.6]{G:JFA}}]
  \label{BCW:prop:traces-on-cones-qd}
  Let $A$ be a $C^*$-algebra.  Then every amenable trace on $CA$ is
  quasidiagonal. In particular if $A$ is nuclear, then all traces on
  $CA$ are quasidiagonal.
\end{proposition}

\begin{proof}
  It is well known that traces of the form $\delta_t\otimes \tau_A$,
  where $\delta_t$ is evaluation at some $t\in(0,1]$ and $\tau_A$ is a
  trace on $A$, generate the Choquet simplex of traces on the cone
  $CA$.\footnote{That is, any trace on $CA$ lies in the
    weak$^*$-closed convex hull of the specified traces.} Since the
  amenable traces on $CA$ form a face (\cite[Lemma 3.4]{Kirchberg94a},
  see also \cite[Proposition 6.3.7]{Brown-Ozawa08}) and the set of
  quasidiagonal traces is a weak$^*$-closed, convex subset of $T(A)$
  (\cite[Proposition 3.5.1]{B:MAMS}), it suffices to show that any
  amenable trace on $CA$ of the form $\delta_t\otimes \tau_A$ for some
  $t\in (0,1]$ and some trace $\tau_A$ on $A$ is quasidiagonal.

  Note too that if $\delta_t\otimes \tau_A$ is an amenable trace on
  $CA$, then $\tau_A$ is amenable on $A$.  This follows from
  \cite[Theorem 3.1.6]{B:MAMS} by checking that the tensor product
  functional $\mu_{\tau_A}$ on the algebraic tensor product $A\odot
  A^{\mathrm{op}}$ given by $\mu_{\tau_A}(a\otimes
  b^{\mathrm{op}})=\tau_A(ab)$ is continuous with respect to the
  minimal tensor product.  Let $g\in C_0(0,1]$ be a positive
  contraction with $g(t)=1$.  Then $\mu_{\tau_A}$ factorizes as
  \begin{equation}
      A\odot A^{\mathrm{op}}
      \xrightarrow{\quad
        a\otimes
        b^{\mathrm{op}}\mapsto(g\otimes a)\otimes (g\otimes
        b)^{\mathrm{op}}\quad}
      CA\odot (CA)^{\mathrm{op}}
      \xrightarrow{\quad\mu_{\delta_t\otimes\tau_A}\quad}
      \mathbb C;
  \end{equation}
  the first of these maps is the tensor product of two c.p.c.\ maps,
  so contractive with respect to the minimal tensor product, while
  contractivity of $\mu_{\delta_t\otimes\tau_A}$ follows from the
  assumption that $\delta_t\otimes\tau_A$ is amenable.

  At this point, if $A$ is not unital, then we can unitize $A$, and
  $\tau_A$ (since the unitization of an amenable trace remains
  amenable).  As a final reduction, by considering the map $CA\to
  C_0((0,t],A)$ given by restriction, and then identifying
  $C_0((0,t],A)$ with $CA$ (by rescaling), we may as well assume that
  $t=1$.  Then \cite[Proposition 3.6]{G:JFA} gives a c.p.c.\ order
  zero map $\Psi:A\rightarrow \mathcal Q_\omega$ (where $\mathcal Q$
  denotes the universal UHF algebra and $\mathcal{Q}_\omega$ its
  ultrapower) such that
  \begin{equation}
    \label{BCW:eq:gabe-cpc}
    \tau_{\mathcal Q_\omega}\big(
    \Psi(a)\Psi(1_A)^{n-1} \big)
    =\tau_A(a), \quad
    a\in A,\ n\in\mathbb N.
  \end{equation}
  By the correspondence between order zero maps from $A$ and
  $^*$-homomorphisms from $CA$ (see \cite[Corollary
  4.1]{Winter-Zacharias09b}) we obtain a $^*$-homomorphism $\psi\colon
  CA\to \mathcal{Q}_\omega$ such that $\psi(\mathrm{id}_{(0,1]}\otimes
  a) = \Psi(a)$ for every $a\in A$.  Then for every $a\in A$ and
  $n\in\mathbb N$,
  \begin{align}\nonumber
      \tau_{\mathcal Q_\omega}\big( \psi(\mathrm{id}_{(0,1]}^n\otimes a) \big)
      &=
        \tau_{\mathcal Q_\omega}\big( \Psi(a)\Psi(1_A)^{n-1} \big)\\
      &=      \tau_A(a)
        =
        ( \delta_1\otimes\tau_A )( \mathrm{id}_{(0,1]}^n\otimes a ).
        \label{BCW:eq:psi-witness}
  \end{align}
 Thus $\psi$
  witnesses the quasidiagonality of the trace $\delta_1\otimes
  \tau_A$.\footnote{\label{FootRef}In the case of non-nuclear $A$, this uses  Lemma~\ref{BCW:Fix} in the next section to show that $\psi$ lifts to a
  c.p.c. map from $CA$ to $\ell^\infty(\mathcal{Q})$. }
\end{proof}

\begin{proof}[{Proof of Theorem \ref{BCW:thm:order0-approximations}}]
  Let $\mathcal{F}\subset A$ be finite and $\epsilon > 0$.  Then there
  is a separable nuclear subalgebra $B$ of $A$ containing
  $\mathcal{F}$. Write $\iota:B\rightarrow A$ for the canonical
  inclusion map.

  Let $\theta\colon B\to CB$ be the c.p.c. order zero map $b\mapsto
  \mathrm{id}_{(0,1]}\otimes b$.  Notice that $CB$ satisfies the
  hypotheses of Theorem~\ref{BCW:thm:QDfactorization}: it is certainly
  separable and nuclear, it is quasidiagonal by a theorem of
  Voiculescu \cite{V:DMJ}, and all of its traces are quasidiagonal by
  Proposition~\ref{BCW:prop:traces-on-cones-qd} (as $CB$ is nuclear,
  all traces are amenable).  Then there are a finite dimensional
  algebra $F$ and c.p.c maps $\psi\colon CB\to F$ and $\phi\colon F\to
  CB$ such that
  \begin{enumerate}
  \item $\| (\phi \circ \psi) (\theta(x)) - \theta(x) \| < \epsilon$;
    
  \item $\phi$ is a convex combination of finitely many contractive
    order zero maps; and
    
  \item $\| \psi\big( \theta(x)\theta(y) \big) - \psi(\theta(x))
    \psi(\theta(y)) \| < \epsilon$;
  \end{enumerate}
  for all $x,y\in \mathcal{F}$.  Let $\eta\colon CB\to B$ be given by
  the point evaluation at $1$ so that $\eta\circ\theta=\mathrm{id}_B$.

  Define a c.p.c. map $\overline{\psi}\colon A\to F$ by extending
  $\psi\circ \theta$ to $A$ (using Arveson's extension theorem) and
  set $\overline{\phi} =\iota\circ \eta \circ \phi\colon F\to A$.
  Then $\overline\phi$ is a convex combination of contractive order
  zero maps (because $\iota\circ\eta$ is a $^*$-homomorphism), $\|
  (\overline\phi \circ \overline\psi) (x) - x \| < \epsilon$ for every
  $x\in \mathcal{F}$, and $\| \overline{\psi}(x) \overline{\psi}(y) \|
  < \epsilon$ if $x,y\in \mathcal{F}$ are orthogonal positive
  elements.
\end{proof}

\begin{remark}
  As with the approximations in \cite{HKW:Adv}, attempting to merge
  the approximations of Theorem \ref{BCW:thm:order0-approximations}
  with the nuclear dimension by additionally asking for a uniform
  bound on the number of summands in the decompositions of $\Phi_i$ as
  a convex combination of order zero maps is very restrictive.  By the
  main result of \cite{C14}, such approximations only exist for $AF$
  $C^*$-algebras.
\end{remark}

\subsection*{Acknowledgements}
S.W. would like to thank Ilan Hirshberg for many helpful conversations
regarding approximations of nuclear $\mathrm{C}^*$-algebras and the
organisers of the Abel symposium for a fantastic conference.  The
authors would also like to thank the referee for a number of helpful
suggestions and comments.

\section{The proof of Proposition~\ref{BCW:prop:traces-on-cones-qd},
  revisited}
  
This section was added in the revision submitted to the editors on 24 August, 2017.  The only changes made to the earlier sections are the addition of footnote \ref{FootRef}, and the updating of references to precise locations in \cite{G:JFA} to reflect the final published version. We have also updated the publication information in the bibliography for \cite{Atkinson15,BBSTWW:MAMS,C14,G:JFA,Tikuisis-White-etal15}.

In Proposition~\ref{BCW:prop:traces-on-cones-qd} we claimed that any
amenable trace $\tau$ on a cone $CA$ over a $\mathrm{C}^*$-algebra is
quasidiagonal. In our proof we produced a $^*$-homomorphism $\psi$
from $CA$ into the ultraproduct $\mathcal Q_\omega$ of the universal
UHF-algebra so that $\tau=\tau_{\mathcal Q_\omega}\circ\psi$. When $A$
is nuclear, the Choi-Effros lifting theorem immediately gives a
c.p.c.\ lift of $\psi$ to $\ell^\infty(\mathcal Q)$.  In the general
case, such a c.p.c.\ lift must be produced in order to obtain
quasidiagonality of $\tau$, and we are sorry that we neglected to do
this in the first version of this article.  Thus for the purposes of all the other
statements in sections 1-3, which refer only
to nuclear $\mathrm{C}^*$-algebras, no additional ingredients are
needed. We complete the proof of
Proposition~\ref{BCW:prop:traces-on-cones-qd} in this section, using
Lemma \ref{BCW:Fix} to give the required lift in the case of general
$A$. 

We're very grateful to Jamie Gabe for pointing out this omission and
to Wilhelm Winter --- the wizard of functional calculus --- for
suggesting the functional calculus trickery used below. 
\bigskip

In the initial stages of the proof of
Proposition~\ref{BCW:prop:traces-on-cones-qd}, we reduced to the
situation where $A$ is a unital $\mathrm{C}^*$-algebra, $\tau_A$ is an
amenable trace on $A$, and we need to show that the trace
$\tau=\delta_1\otimes\tau_A$ on the cone $CA=C_0((0,1])\otimes A$ is
quasidiagonal.  We also (somewhat implicitly) assume that $A$ is
separable. As both amenability and quasidiagonality are local
properties it suffices to treat the case of separable $A$.
Using \cite[Proposition 3.6]{G:JFA} we obtain a c.p.c.\
order zero map $\Psi:A\rightarrow\mathcal Q_\omega$ satisfying
\eqref{BCW:eq:gabe-cpc}.  By construction this c.p.c.\ order zero map
comes with a c.p.c.\ lift to $\ell^\infty(\mathcal Q)$. This is not
recorded explicitly in \cite{G:JFA}, but is readily seen from the
proof: $\Psi$ is of the form $(1_{\mathcal
  Q_\omega}-e)\psi_0(\cdot)(1_{\mathcal Q_\omega}-e)$ for a c.p.c.\
map $\psi_0:A\rightarrow\mathcal Q_\omega$ with a c.p.c.\ lift to
$\ell^\infty(\mathcal Q)$ and a positive contraction $e\in\mathcal
Q_\omega$ (and so $e$ lifts to a representative sequence of positive
contractions in $\ell^\infty(\mathcal Q)$).  We then use the duality
between c.p.c.\ order zero maps and $^*$-homomorphisms from cones
(\cite[Corollary 4.1]{Winter-Zacharias09b}) to obtain a $^*$-homomorphism
$\psi:CA\rightarrow\mathcal Q_\omega$ with $\psi(\id_{(0,1]}\otimes
a)=\Psi(a)$ for all $a\in A$; checking in \eqref{BCW:eq:psi-witness} that
$\psi$ has $\tau_{\mathcal Q_\omega}\circ\psi=\delta_1\otimes\tau_A$.
To complete the proof of Proposition~\ref{BCW:prop:traces-on-cones-qd}
we must show that $\psi$ (which is uniquely determined by $\Psi$) has
a c.p.c.\ lift to $\ell^\infty(\mathcal Q)$.  We do this in the
following lemma, which may well be of use in other situations.

\begin{lemma}\label{BCW:Fix}
  Let $A$ be a separable unital $\mathrm{C}^*$-algebra, and let $B$ be
  a unital $\mathrm{C}^*$-algebra.  Suppose that
  $(\psi_n)_{n=1}^\infty$ is a sequence of c.p.c.\ maps from $A$ into
  $B$ inducing an order zero map $\psi:A\rightarrow B_\omega$. Then
  there exists a sequence of cpc maps $(\phi_n)_{n=1}^\infty$ from
  $C_0((0,1])\otimes A$ into $B$ inducing a $^*$-homomorphism
  $\phi:C_0((0,1])\otimes A\rightarrow B_\omega$ such that
  $\phi(\id\otimes x)=\psi(x)$ for $x\in A$.
\end{lemma}

Before giving the proof of Lemma \ref{BCW:Fix} we record the following
fact which will be used repeatedly.
\begin{lemma}\label{BCW:Hered}
  Let $A$ be a separable unital $\mathrm{C}^*$-algebra, and let $B$ be
  a unital $\mathrm{C}^*$-algebra.  Suppose that
  $(\psi_n)_{n=1}^\infty$ is a sequence of c.p.c.\ maps from $A$ into
  $B$ inducing an order zero map $\psi:A\rightarrow B_\omega$. For a
  contraction $a\in A$, there exists a sequence $(b_n)_{n=1}^\infty$
  of contractions in $B$ so that $(\psi_n(1_A)b_n)_{n=1}^\infty$ and
  $(b_n\psi_n(1_A))_{n=1}^\infty$ both represent $\psi(a)$.
\end{lemma}
\begin{proof}
  By Kirchberg's $\eps$-test (\cite[Lemma A.1]{Kirchberg06}) to prove the
  lemma it suffices to take $\eps>0$ and find a sequence
  $(b_n)_{n=1}^\infty$ of contractions so that
  \begin{equation}
    \lim_{n\rightarrow\omega}
    \| \psi_n(1_A)b_n - \psi_n(a)\|
    \quad\text{and}\quad
    \lim_{n\rightarrow\omega} \| b_n\psi_n(1_A) - \psi_n(a) \|
    < \eps.
  \end{equation}
  Recall from \cite[Theorem 3.3]{Winter-Zacharias09b} that the supporting
  $^*$-homomorphism of $\psi$ is a $^*$-homomorphism $\pi:A\rightarrow
  \mathcal M(C^*(\psi(A)))$ satisfying
  $\psi(x)=\pi(x)\psi(1_A)=\psi(1_A)\pi(x)$, for $x\in A$.  Then, for
  $f\in C_0((0,1])_+$, $f(\psi)$ is defined by
  $f(\psi)(x)=\pi(x)f(\psi(1_A))$, for $x\in A$; this is a c.p.\ order
  zero map (which is contractive when $f$ is). Taking $f(t)=t^r$,
  where $0<r<1$, gives
 \begin{equation}
    \psi(x) = \pi(x)\psi(1_A)^r\psi(1_A)^{1-r}
    = \psi^r(a)\psi(1_A)^{1-r}
    = \psi(1_A)^{1-r}\psi^r(x),\quad
    x\in A.
  \end{equation}
  Fix $0<r<1$ small enough so that $\sup_{t\in
    [0,1]}|t^{1-r}-t|<\eps$.  For a contraction $a\in A$, let
  $(b_n)_{n=1}^\infty$ be a sequence of positive contractions in $B$
  representing the contraction $\psi^r(a)$. Then
  \begin{align}
    \nonumber\lim_{n\rightarrow\omega}
    \| \psi_n(1_A)b_n-\psi_n(a)\|
    &= \lim_{n\rightarrow\omega} \| \psi_n(1_A)b_n -
      \psi_n(1_A)^{1-r}b_n\|\\
    & \leq \lim_{n\rightarrow\omega} \| \psi_n(1_A) -
      \psi_n(1_A)^{1-r} \|
      < \eps.
  \end{align}
  Likewise
  \begin{equation}
    \lim_{n\rightarrow\omega}\|b_n\psi_n(1_A)-\psi_n(a)\|<\eps,
  \end{equation}
  as required.
\end{proof}

\begin{proof}[Proof of Lemma \ref{BCW:Fix}.]
  We use Kirchberg's $\eps$-test (\cite[Lemma A.1]{Kirchberg06}) to handle
  the reindexing and allow us to prove a slightly weaker statement. To
  set this up, for each $n$, let\footnote{There is no dependence on
    $n$ in the definition of $X_n$. We use this notation for
    consistency with the usual formulation of Kirchberg's
    $\eps$-test.} $X_n$ be the set of c.p.c.\ maps $C_0((0,1])\otimes
  A\rightarrow B$, let $(a_i)_{i=1}^\infty$ be a sequence which is
  dense in the positive contractions of $A$ and $(f_i)_{i=1}^\infty$ a
  sequence which is dense in the positive contractions of
  $C_0((0,1])$.  For $i,i_1,i_2,j_1,j_2\in\mathbb N$, define
  $r^{(i_1,i_2,j_1,j_2)}_n,s^{(i)}_n:X_n\rightarrow[0,\infty)$ by
  \begin{equation}
    r^{(i_1,i_2,j_1,j_2)}_n(\phi_n)
    = \| \phi_n(f_{i_1}\otimes a_{i_2}) \phi_n(f_{j_1}\otimes a_{j_2})
    - \phi_n(f_{i_1}f_{j_1}\otimes a_{i_2}a_{j_2}) \|
  \end{equation}
  and 
  \begin{equation}
    s^{(i)}_n(\phi_n) = \|\phi_n(\id\otimes a_i) - \psi_n(a_i)\|.
  \end{equation}
  Then a sequence $(\phi_n)_{n=1}^\infty\in\prod_{n=1}^\infty X_n$ induces a
  $^*$-homomorphism $\phi:C_0((0,1])\otimes A\rightarrow B_\omega$ if
  and only if
  $\lim_{n\rightarrow\omega}r^{(i_1,i_2,j_1,j_2)}(\phi_n)=0$ for all
  $i_1,i_2,j_1,j_2\in\mathbb N$ and has $\phi(\id\otimes x)=\psi(x)$
  for all $x\in A$ if and only if
  $\lim_{n\rightarrow\omega}s^{(i)}_n(\phi_n)=0$ for all $i\in\mathbb
  N$.  Thus, by Kirchberg's $\eps$-test we can fix $\eps>0$ and
  $i_0\in\mathbb N$, and it suffices to find a sequence
  $(\phi_n)_{n=1}^\infty\in\prod_{n=1}^\infty X_n$ such that
  \begin{equation}\label{BCW:EpsTest}
    \lim_{n\rightarrow\omega}r^{(i_1,i_2,j_1,j_2)}_n(\phi_n)\leq \eps
    \quad\text{ and }\quad
    \lim_{n\rightarrow\omega}s^{(i)}_n(\phi_n)\leq\eps
  \end{equation}
  for $i,i_1,i_2,j_1,j_2=1,\dots,i_0$.

  For $\delta>0$ to be chosen later, define $g_\delta\in C_0((0,1])$
  by
  \begin{equation}
    g_\delta(t)=
    \begin{cases}
      \delta^{-2}t,&0\leq t\leq\delta\\
      t^{-1},& t>\delta
    \end{cases}.
\end{equation}
  Note that 
  \begin{equation}\label{BCW:GEst}
    0\leq g_\delta(t)t\leq 1,\quad t\in [0,\infty),
    \quad\text{ and }\quad
    \sup_{t\in [0,1]}|(g_\delta(t)t-1)t|\leq\delta.
  \end{equation}
  Fix $m\in\mathbb N$ with the property that
  \begin{equation}\label{BCW:Choicem}
  |f_i(x)-f_i(y)|\leq\eps/3\text{ whenever }|x-y|\leq 1/m
  \end{equation}
  for all for $i\leq i_0$.  Let $h_0,\dots,h_m$ be standard
  `saw-tooth' partition of unity corresponding to the division of
  $[0,1]$ into $m$ intervals.  That is, $h_i$ is (the restriction to
  $[0,1]$ of) the piecewise affine function satisfying
  $h_i((i-1)/m)=0$, $h_i(i/m)=1$, $h_i((i+1)/m)=0$, and defined to be
  affine on the intervals $[(i-1)/m,i/m]$ and $[i/m,(i+1)/m]$ and $0$
  outside the interval $[(i-1)/m,(i+1)/m]$. In this way, for any
  scalars $\alpha_0,\dots,\alpha_m$, the element
  $f=\sum_{l=0}^m\alpha_lh_l$ is the piecewise affine function in
  $C([0,1])$ with $f(j/m)=\alpha_j$ and affine on the intervals
  $[j/n,(j+1)/n]$.

  To simplify notation we will write $e_n$ for $\psi_n(1_A)$.  For
  each $n$, we define a map $\phi_n:C_0((0,1])\otimes A\rightarrow B$
  by
  \begin{align}
    \phi_n(f\otimes a) = \sum_{l=0}^m
    f(l/m)h_l(e_n)^{1/2}g_\delta(e_n)^{1/2}\psi_n(a) g_\delta(e_n)^{1/2}h_l(e_n)^{1/2}
  \end{align}
  for $f\in C_0((0,1])$ and $a\in A$.  This certainly extends to a
  linear map, and as $\psi_n$ is completely positive, so too is each
  map in the sum defining $\phi_n$, and so $\phi_n$ is completely
  positive.  Now for a positive contraction $f\in C_0((0,1])$ we have
  \begin{align}
    \nonumber
    \phi_n(f\otimes 1_A)
    &= \sum_{l=0}^mf(l/m)g_\delta(e_n)h_l(e_n)e_n\\
    &\leq g_\delta(e_n)e_n\sum_{l=0}^m h_l(e_n)\leq 1_B,
  \end{align}
  as $0\leq g_\delta(t)t\leq 1$ and $\sum_{l=0}^mh_l(t)=1$ for all
  $t\in [0,1]$. Since $f$ is an arbitrary positive contraction, it
  follows that $\phi_n$ is contractive (without any condition on
  $\delta$). It remains to show that we can choose $\delta$ so that
  the sequence $(\phi_n)_{n=1}^\infty\in\prod_{n=1}^\infty X_n$ verifies the two
  conditions of (\ref{BCW:EpsTest}).

  Since the sequence $(\psi_n)_{n=1}^\infty$ induces an order zero
  map, we have
  \begin{equation}\label{BCW:OZCom}
    \lim_{n\rightarrow\omega} \|[e_n,\psi_n(a)]\| = 0,\quad a\in A.
  \end{equation}
  Therefore
  \begin{equation}
    \lim_{n\rightarrow\omega}
    \Big\| \phi_n(\id\otimes a) - \psi_n(a)g_\delta(e_n) \sum_{l=0}^m
    \frac{l}{m}h_l(e_n) \Big\| = 0.
  \end{equation}
  Now $\sum_{l=0}^m\frac{l}{m}h_l(t)=t$, so 
  \begin{equation}
    \lim_{n\rightarrow\omega}\|\phi_n(\id\otimes a)-\psi_n(a)g_\delta(e_n)e_n\|=0.
  \end{equation}
  Fix a contraction $a\in A$, and by Lemma \ref{BCW:Hered}, let
  $(b_n)_{n=1}^\infty$ be a sequence of contractions in $B$ so that
  $(b_ne_n)_{n=1}^\infty$ represents $\psi(a)$.  Using (\ref{BCW:GEst})
  for the second estimate below, we have
  \begin{equation}
    \lim_{n\rightarrow\omega}\|\psi_n(a)g_\delta(e_n)e_n-\psi_n(a)\|
    \leq \lim_{n\rightarrow\omega}\|b_n\|\|e_ng_\delta(e_n)e_n-e_n\|\leq\delta.
  \end{equation}
  In this way
  \begin{equation}
    \lim_{n\rightarrow\omega}\|\phi_n(\id\otimes a)-\psi_n(a)\|\leq \delta
  \end{equation}
  for any contraction $a\in A$.  Then, provided we ensure
  $\delta<\eps$, we get $\lim_{n\rightarrow\omega} s^{(i)}_n (\phi_n)
  \leq \eps$ for all $i$.

  Now we show `almost multiplicativity'. For this, fix
  $i_1,i_2,j_1,j_2\leq i_0$.  Using (\ref{BCW:OZCom}) and the fact that
  $h_kh_l=0$ for $|k-l|\geq 2$ we have
  \begin{align}
    &\lim_{n\rightarrow\omega}
      \| \phi_n(f_{i_1}\otimes a_{i_2}) \phi_n(f_{j_1}\otimes a_{j_2})
      - \phi_n(f_{i_1}f_{j_1}\otimes a_{i_2}a_{j_2}) \|\\
    \nonumber
    =&\lim_{n\rightarrow\omega} \Big\| \sum_{|k-l|\leq 1}
       f_{i_1}(k/m)f_{j_1}(l/m)h_k(e_n)h_l(e_n)
       g_\delta(e_n)\psi_n(a_{i_2}) g_\delta(e_n)\psi_n(a_{j_2})\\
    \nonumber
    &\qquad\quad- \sum_{l=0}^m
      f_{i_1}(l/m)f_{j_1}(l/m) h_l(e_n)g_\delta(e_n)
      \psi_n(a_{i_2}a_{j_2}) \Big\|.
  \end{align} 
  Using Lemma \ref{BCW:Hered}, we can find a sequence $(b_n)_{n=1}^\infty$
  of contractions in $B$ so that $(e_nb_n)_{n=1}^\infty$ represents
  $\psi(a_{i_2})$. In this way (\ref{BCW:GEst}) gives
  \begin{equation}
    \lim_{n\rightarrow\omega}
    \| g_\delta(e_n)\psi_n(a_{i_2})\|
    = \lim_{n\rightarrow\omega} \| g_\delta(e_n)e_nb_n \|
    \leq 1. 
  \end{equation}
  Likewise
  $\lim_{n\rightarrow\omega}\|g_\delta(e_n)\psi_n(a_{j_2})\|\leq 1$.
  As, for each $l$, $\sum_{k=l-1}^{l+1}h_k$ acts as a unit on
  $h_l$,\footnote{When $l=0$ or $l=m$ there are only two terms in this
    sum, but the result still holds.} we may use the estimates above
  to obtain
  \begin{align}\label{BCW:eq:m-m}
    \lim_{n\rightarrow\omega}
    &\Big\| \sum_{|k-l|\leq 1}
      f_{i_1}(k/m)f_{j_1}(l/m)h_k(e_n)h_l(e_n)
      g_\delta(e_n)\psi_n(a_{i_2}) g_\delta(e_n)\psi_n(a_{j_2})\\
    \nonumber
    &\quad -\sum_{l=0}^m
      f_{i_1}(l/m)f_{j_1}(l/m)h_l(e_n) g_\delta(e_n)\psi_n(a_{i_2})
      g_\delta(e_n)\psi_n(a_{j_2}) \Big\|\\
    \nonumber
    \leq \lim_{n\rightarrow\omega}
    &\Big\| \sum_{l=0}^m
      f_{j_1}(l/m)h_l(e_n)
      \sum_{k=l-1}^{l+1} \big( f_{i_1}(k/m) - f_{i_1}(l/m) \big)
      h_k(e_n) \Big\|.
  \end{align}
  Using the choice of $m$ in (\ref{BCW:Choicem}), observe that \( |
  \sum_{k=l-1}^{l+1} \big( f_{i_1}(k/m) - f_{i_1}(l/m) \big) h_k |
  \leq (\varepsilon/3) \sum_{k=l-1}^{l+1} h_k \) (in $C([0,1])$) for
  every $l = 0, \dots, m$, and therefore
  \begin{align}
    &\Big|\sum_{l=0}^m f_{j_1}(l/m)h_l \sum_{k=l-1}^{l+1} \big(
      f_{i_1}(k/m) - f_{i_1}(l/m) \big) h_k\Big|
      \leq
      (\varepsilon/3) \sum_{l=0}^m f_{j_1}(l/m) h_l \leq \varepsilon/3,
  \end{align}
  as $f_{j_1}$ is a contraction.  This shows that the last limit in
  \eqref{BCW:eq:m-m} is bounded above by $\varepsilon/3$.

  Using this, the order zero identity
  $\psi(1_A)\psi(a_{i_2}a_{j_2})=\psi(a_{i_1})\psi(a_{j_2})$ (see
  \cite[(1.3)]{BBSTWW:MAMS}, for example) and (\ref{BCW:OZCom}) again, we
  have
  \begin{align}\label{BCW:E1}
    &\lim_{n\rightarrow\omega}
      \| \phi_n(f_{i_1}\otimes a_{i_2}) \phi_n(f_{j_1}\otimes a_{j_2})
      - \phi_n(f_{i_1}f_{j_1}\otimes a_{i_2}a_{j_2}) \|\\
    \nonumber
    &\leq\lim_{n\rightarrow\omega}
      \Big\| \sum_{l=0}^m f_{i_1}(l/m)f_{j_1}(l/m)h_l(e_n) g_\delta(e_n)^2
      e_n \psi_n(a_{i_2}a_{j_2})\\
    \nonumber
    &\qquad\qquad - \sum_{l=0}^m f_{i_1}(l/m)f_{j_1}(l/m)h_l(e_n)
      g_\delta(e_n)\psi_n(a_{i_2}a_{j_2}) \Big\| + \eps/3.
  \end{align}

  Let $f_{i_1,j_1}(t)=\sum_{l=0}^mf_{i_1}(l/m)f_{j_1}(l/m)h_l(t)$, so
  that $f_{i_1,j_1}$ is a contraction in $C_0((0,1])$.  We aim to
  control
  \begin{equation}
    \lim_{n\rightarrow\omega}
    \big\| f_{i_1,j_1}(e_n)\big(g_\delta(e_n)^2e_n - g_\delta(e_n) \big)
    \psi_n(a_{i_2}a_{j_2}) \big\|.
  \end{equation}
  Use Lemma \ref{BCW:Hered} to find a sequence $(b_n)_{n=1}^\infty$ of
  contractions in $B$ so that $(e_nb_n)_n$ represents
  $\psi(a_{i_2}a_{j_2})$.  As (\ref{BCW:GEst}) gives $0\leq
  g_\delta(e_n)e_n\leq 1$ in $B$, so
  $\|g_\delta(e_n)^2e_n^2-g_\delta(e_n)e_n\|\leq\frac{1}{\sqrt{2}}-\frac{1}{2}$,
  giving
  \begin{align}
    \lim_{n\rightarrow\omega}
      \big\| \big(g_\delta(e_n)^2e_n - g_\delta(e_n)\big)
      \psi_n(a_{i_2}a_{j_2}) \big\|
    &=
    \lim_{n\rightarrow\omega}
      \big\| \big(g_\delta(e_n)^2e_n^2 - g_\delta(e_n)e_n\big)b_n
      \big\|
    \\\nonumber
    &\leq \frac{1}{\sqrt{2}}-\frac{1}{2},
  \end{align}
  \emph{independently} of the value of $\delta$ used in the definition
  of $g_\delta$.  Fix a polynomial function $p_{i_1,j_1}\in
  C_0((0,1])$ (so with no constant term) so that $\|p_{i_1,j_1} -
  f_{i_1,j_1}\| \leq \eps/3(\frac{1}{\sqrt{2}} -
  \frac{1}{2})$. Therefore, using the fact that
  $(e_nb_n)_{n=1}^\infty$ represents the same sequence as
  $(\psi_n(a_{i_2}a_{j_2}))_{n=1}^\infty$, we have
  \begin{align}\label{BCW:E2}
    &\lim_{n\rightarrow\omega}
      \big\| f_{i_1,j_1}(e_n) \big(g_\delta(e_n)^2e_n - g_\delta(e_n)\big)
      \psi_n(a_{i_2}a_{j_2})\big\|\\
    \nonumber\leq
    &\lim_{n\rightarrow\omega}
      \big\| p_{i_1,j_1}(e_n)
      \big(g_\delta(e_n)^2e_n^2-g_\delta(e_n)e_n\big)b_n \big\| + \eps/3,
  \end{align}
  again independent of the choice of $\delta$.  Finally, as
  $\sup_{t\in [0,1]}\big| t\big(g_\delta(t)^2t^2-g_\delta(t)t\big)
  \big|\leq \delta$ (from (\ref{BCW:GEst})) and $p_{i_1,j_1}(t)$ factors
  as $q_{i_1,j_1}(t)t$ for some polynomial $q_{i_1,j_1}(t)$ (this time
  with a possible constant term) we can now choose $\delta \leq
  \eps$ so that
  \begin{equation}\label{BCW:E3}
    \lim_{n\rightarrow\omega}
    \big\| p_{i_1,j_1}(e_n)
    \big(g_\delta(e_n)^2e_n^2-g_\delta(e_n)e_n\big) \big\|
  \leq \eps/3
  \end{equation}
  for all $i_1,j_1\leq i_0$.  Putting (\ref{BCW:E3}) together with
  (\ref{BCW:E1}) and (\ref{BCW:E2}), we obtain
  \begin{equation}
    \lim_{n\rightarrow\omega}
    r_n^{(i_1,i_2,j_1,j_2)}(\phi_n) \leq \eps,
  \end{equation}
for all $i_1,i_2,j_1,j_2\leq i_0$,  as required.
\end{proof}

\newcommand{\etalchar}[1]{$^{#1}$}
\providecommand{\bysame}{\leavevmode\hbox to3em{\hrulefill}\thinspace}

\providecommand{\href}[2]{#2}

\end{document}